\documentclass[preprint,11pt]{elsarticle}

\usepackage{amsmath,amssymb,amsthm,mathtools}
\usepackage{microtype}
\usepackage{enumitem}
\usepackage[hidelinks]{hyperref}
\setlist[enumerate]{leftmargin=*,itemsep=2pt,topsep=4pt}

\newtheorem{theorem}{Theorem}[section]
\newtheorem{lemma}[theorem]{Lemma}
\newtheorem{proposition}[theorem]{Proposition}

\theoremstyle{remark}
\newtheorem{remark}[theorem]{Remark}

\DeclareMathOperator{\supp}{supp}
\DeclareMathOperator{\wt}{wt}
\newcommand{\calH}{\mathcal H}
\newcommand{\calC}{\mathcal C}

\journal{Journal of Combinatorial Theory, Series A}
\begin{document}
\begin{frontmatter}

\title{Asymptotically Sharp Length Thresholds for \texorpdfstring{$q$}{q}-ary Codes with Distances \texorpdfstring{$2$}{2} and \texorpdfstring{$4$}{4}}
\author[ncu]{Valjakas Djaljapayan\corref{cor1}}
\ead{115281004@cc.ncu.edu.tw}

\author[ntnu]{Yen-Chi Roger Lin}
\ead{yclin@math.ntnu.edu.tw}

\author[ncu]{Wei-Hsuan Yu}
\ead{whyu@math.ncu.edu.tw}

\cortext[cor1]{Corresponding author}

\affiliation[ncu]{organization={Department of Mathematics, National Central University},
                  country={Taiwan}}
\affiliation[ntnu]{organization={Department of Mathematics, National Taiwan Normal University},
                   country={Taiwan}}

\begin{abstract}
For an integer $q\ge5$, let $A_q(n,\{2,4\})$ be the largest cardinality of a subset
of a $q$-ary Hamming space of length $n$ in which every two distinct words have
distance $2$ or $4$.  Set
\[
N(q)=\left\lceil\max\left\{2q+1,
 \frac{17+\sqrt{8q^2+96q-95}}2\right\}\right\rceil.
\]
We prove that $A_q(n,\{2,4\})=1+\binom n2$ for every $n\ge N(q)$.
In particular, $N(q)=2q+1$ for $q\ge26$.  An explicit code of size
$q(n-2)$ shows that the least eventual threshold $N_*(q)$ satisfies
$N_*(q)\ge2q-1$; consequently
$2q-1\le N_*(q)\le2q+1$ for every $q\ge26$.
The argument applies to arbitrary integer alphabet sizes, without a linearity
assumption.  Its main ingredients are synchronization of symbols on intersecting
supports, weighted bounds for $2$-intersecting families of four-sets, and a
joint estimate for the weight-two and weight-four layers in a common-three-core
configuration.  No computational verification is required.
\end{abstract}

\begin{keyword}
few-distance codes \sep Hamming space \sep intersecting families \sep
matching number \sep Krawtchouk polynomials \sep extremal combinatorics
\MSC[2020] 94B65 \sep 05D05 \sep 05C35
\end{keyword}
\end{frontmatter}

\section{Introduction and main results}\label{sec:intro}

Let $\Sigma$ be a finite alphabet of cardinality $q$, equipped with a
distinguished symbol $0$.  The \emph{Hamming distance} $d(x,y)$ between
$x,y\in\Sigma^n$ is the number of coordinates on which they differ.  For
$D\subseteq\{1,\ldots,n\}$, let
\[
 A_q(n,D)=\max\bigl\{|\calC|:\ \calC\subseteq\Sigma^n,
                  \ d(x,y)\in D\text{ whenever }x\ne y\in\calC\bigr\}.
\]
No linear or additive structure is imposed on $\calC$ or $\Sigma$.
The standard code consisting of the zero word and all words with a fixed
nonzero symbol in exactly two positions gives
\begin{equation}\label{eq:standard-lower}
 A_q(n,\{2,4\})\ge1+\binom n2.
\end{equation}
Indeed, two distinct nonzero words in this code are at distance two if their
supports intersect and at distance four if their supports are disjoint.

Prescribed-distance codes are classical objects in coding theory and algebraic
combinatorics, originating in the metric viewpoint of Hamming~\cite{Hamming50}.  Delsarte's association-scheme method gives a general framework
for their distance distributions~\cite{Delsarte73,BannaiIto84,BCN89}; refinements
involving the Terwilliger algebra yield nonbinary semidefinite bounds
\cite{Sch05,GST06}.  The links with two-distance sets and finite geometry
are developed in~\cite{DGS77,EFIN87,CK86}.  We refer to
\cite{MS77,vanLint99,HP03} for background on Hamming spaces.

The two-distance problem with allowed distances $\{2,4\}$ appears in the work
of Boyvalenkov, Delchev, Zinoviev, and Zinoviev~\cite{BDZZ21}, which also
provides tables of finite bounds.  For smaller alphabets, binary
few-distance results have been obtained by Barg, Glazyrin, Kao, Lai, Tseng,
and Yu~\cite{Barg24}; see also the related constructions in\cite{LRV23}.  Here we investigate how large the length must be,
as a function of $q$, before the quadratic construction in
\eqref{eq:standard-lower} is optimal.  The dependence on $q$ matters:
when $q\ge5$, the quadratic construction is not always optimal at smaller
lengths.

For $q\ge5$, define
\begin{equation}\label{eq:threshold}
 N(q):=\left\lceil
 \max\left\{2q+1,\frac{17+\sqrt{8q^2+96q-95}}{2}\right\}
 \right\rceil.
\end{equation}
Our principal results are as follows.

\begin{theorem}[Exact value at large length]\label{thm:main}
For every integer $q\ge5$ and every integer $n\ge N(q)$,
\begin{equation}\label{eq:main}
 A_q(n,\{2,4\})=1+\binom n2.
\end{equation}
In particular, the equality holds for $n\ge2q+1$ whenever $q\ge26$.
\end{theorem}

Define $N_*(q)$ to be the smallest integer such that
\eqref{eq:main} holds for every integer $n\ge N_*(q)$; the existence of
$N_*(q)$ follows from Theorem~\ref{thm:main}.

\begin{theorem}[Nearly exact threshold]\label{thm:threshold}
For every integer $q\ge5$,
\[
 N_*(q)\ge2q-1.
\]
For every $q\ge26$, moreover,
\begin{equation}\label{eq:bracket}
 2q-1\le N_*(q)\le2q+1.
\end{equation}
In particular, $N_*(q)=2q+O(1)$, and the leading constant is best
possible.
\end{theorem}

The method begins by choosing a codeword and independently relabeling the
alphabet in each coordinate so that this word becomes zero.  The remaining
words have weight two or four.  Their supports obey strong intersection
restrictions, but the supports alone do not encode the entire problem:
multiple assignments of nonzero symbols can occur on a single support.
The central observation is that an exact two-coordinate intersection between
distinct four-supports forces agreement of the symbols on the intersection.
We combine this with the complete nontrivial-intersection theorem
of Ahlswede and Khachatrian~\cite{AK96} and an elementary matching
analysis derived from the Erd\H{o}s--Gallai theorem~\cite{EG59}.
The set-theoretic perspective originates with the Erd\H{o}s--Ko--Rado theorem
\cite{EKR61}, and developed through nontrivial intersection bounds
\cite{HM67,Frankl78,AK97} and algebraic intersection techniques
\cite{FW81,FT18}.  For the associated graph-matching viewpoint, see
\cite{LovaszPlummer86}.

A large family of four-supports is either controlled by a common two-core,
or has bounded weighted size, or exhibits a common three-core.  The
last case requires a joint bound on the weight-two and weight-four layers;
separate estimates would lose the optimal leading constant in $N(q)$.
The detailed structure of the proof is reflected in the order of the
sections: elementary synchronization, local multiplicities, common-core
geometry, structural alternatives, and finally the two theorems.

\section{Normalization and weight-two supports}\label{sec:weight2}

Fix $q\ge5$ throughout, and put
\begin{equation}\label{eq:mdef}
 m:=q-1.
\end{equation}
Any permutation of the alphabet performed independently in each coordinate
preserves Hamming distance.  Thus we may arrange that $0\in\calC$.  Every
nonzero codeword then has Hamming weight two or four, so
\begin{equation}\label{eq:layers}
 \calC=\{0\}\mathbin{\dot\cup}C_2\mathbin{\dot\cup}C_4,
 \qquad C_j=\{x\in\calC:\wt(x)=j\}.
\end{equation}
Here $\supp(x)=\{i:x_i\ne0\}$ and $\wt(x)=|\supp(x)|$.
If a set $S\subseteq[n]$ occurs as the support of words in one of the two
layers, its \emph{multiplicity} is the number of words on that support.

\begin{lemma}[Synchronization at a single coordinate]\label{lem:single}
If $x\in C_2$, $y\in C_2\cup C_4$, and
$\supp(x)\cap\supp(y)=\{i\}$, then $x_i=y_i$.
\end{lemma}
\begin{proof}
Let $\delta=1$ if $x_i\ne y_i$ and $\delta=0$ otherwise.  Since
$i$ is the only common nonzero coordinate,
\[
 d(x,y)=\wt(x)+\wt(y)-2+\delta.
\]
The right-hand side is $2+\delta$ or $4+\delta$, according as
$y$ has weight two or four.  Both $3$ and $5$ are forbidden, so
$\delta=0$.
\end{proof}

Call an occurring two-support \emph{isolated} if it is disjoint from every
other occurring two-support, and \emph{active} otherwise.

\begin{lemma}[Pure weight-two bound]\label{lem:pure}
An isolated two-support occurs at most $m$ times, while an active
two-support occurs exactly once.  Consequently, if $B\subseteq\Sigma^N$
is a collection of weight-two words with all mutual distances in
$\{2,4\}$ and $N\ge q$, then
\begin{equation}\label{eq:pure}
 |B|\le\binom N2.
\end{equation}
\end{lemma}
\begin{proof}
Two distinct words on the same two-support must differ in both its
coordinates, since their distance is at most two and cannot be one.
Projection onto either coordinate is therefore injective into the
$m$ nonzero symbols.  This proves the first multiplicity bound.

If a support $S$ is active, there is a distinct two-support $T$
with $|S\cap T|=1$.  Two different words on $S$ would disagree at
the coordinate $S\cap T$, whereas Lemma~\ref{lem:single}, applied to
both against a fixed word on $T$, forces agreement.  Thus $S$
has multiplicity one.

Let $h$ be the number of isolated supports in $B$.  They are pairwise
disjoint and occupy $2h$ coordinates; all active supports avoid
these coordinates.  Therefore
\[
 |B|\le mh+\binom{N-2h}{2}.
\]
Using $h\le N/2$ and $N\ge q=m+1$, we obtain
\[
 \binom N2-mh-\binom{N-2h}{2}
 =h(2N-2h-q)\ge h(N-q)\ge0.
\]
This proves \eqref{eq:pure}.
\end{proof}

\begin{lemma}[A linear estimate for $C_2$]\label{lem:C2-linear}
If $C_4\ne\varnothing$, then
\begin{equation}\label{eq:C2-linear}
 |C_2|\le4n+4m-14.
\end{equation}
\end{lemma}
\begin{proof}
Fix $z\in C_4$.  Every occurring two-support meets $\supp(z)$,
otherwise its codeword would be at distance six from $z$.  There are
\[
 \binom n2-\binom{n-4}{2}=4n-10
\]
two-subsets meeting a fixed four-set.  At most four occurring
two-supports are isolated, since they are disjoint and each meets
$\supp(z)$.  By Lemma~\ref{lem:pure}, each isolated support may
contribute at most $m-1$ extra words beyond one word per support.
Hence $|C_2|\le4n-10+4(m-1)$.
\end{proof}

\section{Multiplicity bounds on four-supports}\label{sec:local}

Let
\begin{equation}\label{eq:H}
 \begin{aligned}
 \calH&:=\{\supp(x):x\in C_4\}\subseteq\binom{[n]}4,\\
 \mu(S)&:=|\{x\in C_4:\supp(x)=S\}|\qquad(S\in\calH).
 \end{aligned}
\end{equation}
In particular, $|C_4|=\sum_{S\in\calH}\mu(S)$.

\begin{lemma}[Multiplicity on one four-support]\label{lem:local}
Every $S\in\calH$ has multiplicity
\begin{equation}\label{eq:local}
 \mu(S)\le m^2.
\end{equation}
\end{lemma}
\begin{proof}
Restrict the codewords supported on $S$ to $S$.  They form a code
$B\subseteq\Omega^4$, where $\Omega=\Sigma\setminus\{0\}$ has $m\ge4$
symbols and all nonzero mutual distances are in $\{2,4\}$.
For $i\in\{0,2,4\}$, let
\[
 a_i=\frac{1}{|B|}
       |\{(x,y)\in B^2:d(x,y)=i\}|.
\]
Thus $a_0=1$ and $|B|=1+a_2+a_4$.

For $j=1,3$, consider the length-four $m$-ary Krawtchouk polynomial
\[
 K_j(i)=\sum_{s=0}^j(-1)^s(m-1)^{j-s}
             \binom{i}{s}\binom{4-i}{j-s}.
\]
The Hamming positivity inequality is
\begin{equation}\label{eq:positive}
 \sum_{i\in\{0,2,4\}}a_iK_j(i)\ge0\quad(j=1,3).
\end{equation}
For completeness, choose an orthonormal basis
$\phi_0\equiv1,\phi_1,\ldots,\phi_{m-1}$ of the real-valued functions
on $\Omega$ for the normalized counting inner product.  Then
\[
 \sum_{r=1}^{m-1}\phi_r(a)\phi_r(b)
 =m\mathbf1_{\{a=b\}}-1.
\]
For $j=1,3$, expand the nonnegative quantity
\[
 \frac1{|B|}\sum_{\substack{I\subseteq[4]\\|I|=j}}
  \sum_{\alpha\in\{1,\ldots,m-1\}^{I}}
  \left(\sum_{x\in B}\prod_{i\in I}\phi_{\alpha_i}(x_i)\right)^2.
\]
A pair of words at distance $i$ contributes $K_j(i)$, proving
\eqref{eq:positive}.  Thus no linearity of $B$ is needed.

The relevant values are
\[
\begin{array}{c|ccc}
 &i=0&i=2&i=4\\ \hline
K_1(i)&4(m-1)&2(m-2)&-4\\
K_3(i)&4(m-1)^3&-2(m-2)(m-1)&-4
\end{array}
\]
and hence
\begin{align*}
 4(m-1)+2(m-2)a_2-4a_4&\ge0,\\
 4(m-1)^3-2(m-2)(m-1)a_2-4a_4&\ge0.
\end{align*}
Multiply the first inequality by $m-3\ge0$ and add the second.
After simplification,
\[
 4(m-2)\bigl(m^2-1-a_2-a_4\bigr)\ge0.
\]
Since $m\ge4$, this yields $|B|=1+a_2+a_4\le m^2$.
\end{proof}

\begin{lemma}[Exact overlaps]\label{lem:overlap}
If $S,T\in\calH$ are distinct and $|S\cap T|=2$, then
\begin{equation}\label{eq:overlap}
 \mu(S),\mu(T)\le m.
\end{equation}
If, in addition, $U\in\calH$ satisfies $|S\cap U|=2$ and
$S\cap U\ne S\cap T$, then $\mu(S)=1$.
\end{lemma}
\begin{proof}
Let $P=S\cap T$.  For any $x,y\in C_4$ supported on $S,T$,
respectively, the four coordinates of $S\triangle T$ already
contribute four to $d(x,y)$.  Hence $x|_P=y|_P$.
In particular, every word on $S$ has the same symbols on $P$.
Two distinct such words can differ only on the other two coordinates
of $S$, and must differ in both.  Projection to either of these
coordinates is injective into an $m$-symbol alphabet.  Therefore
$\mu(S)\le m$, and symmetry gives the same bound on $T$.

The second exact overlap with $U$ fixes the symbols of every word on
$S$ on another two-set $S\cap U$.  Since the two fixed pairs are
distinct, their union contains at least three of the four
coordinates of $S$.  Two different words on $S$ would have distance
at most one, which is impossible.
\end{proof}

\begin{lemma}[Intersecting supports]\label{lem:2int}
Any two distinct members of $\calH$ intersect in at least two coordinates.
\end{lemma}
\begin{proof}
For $S,T\in\calH$, choose words $x,y$ with these supports, and put
$t=|S\cap T|$.  If $\delta$ counts symbol disagreements on $S\cap T$,
then $d(x,y)=8-2t+\delta\le4$, so $t\ge2$.
\end{proof}

\section{Three-coordinate geometry and common cores}\label{sec:threecore}

The next elementary lemmas quantify how much multiplicity can remain
when several four-supports share three coordinates.

\begin{lemma}[Distance-two codes in three coordinates]\label{lem:distance2-3}
Let $\Omega$ be an alphabet of size $m\ge4$.  If $B\subseteq\Omega^3$
has distance two between any distinct pair of words, then $|B|\le m$.
\end{lemma}
\begin{proof}
If $|B|\le1$ there is nothing to prove.  Choose two words and
normalize them, by coordinatewise alphabet permutations, to
$000$ and $110$.  A third word at distance two from both
has one of the following forms:
\[
 (a,b,0)\quad(a,b\notin\{0,1\}),\qquad
 (1,0,c),\qquad(0,1,c)\quad(c\ne0).
\]
Words of the first form have pairwise distance two only if their
first coordinates, and separately their second coordinates, are
distinct.  There are therefore at most $m-2$ such words.
At most one word of each of the latter two forms can occur, and
neither can coexist with a word of the first form, because that
would give distance three.  Thus $|B|\le\max\{m,4\}=m$.
\end{proof}

\begin{lemma}[Three-coordinate projection]\label{lem:projection}
Let $\Omega$ have $m\ge4$ symbols, and let $B\subseteq\Omega^4$
have all nonzero distances in $\{2,4\}$.  Suppose the projection of
$B$ onto a specified three-set of coordinates contains no pair of
words at distance three.  Then $|B|\le m$.
\end{lemma}
\begin{proof}
The projection is injective: otherwise two different words of $B$
would differ only in the omitted coordinate.  Denote the projected
set by $Y$, and let $f:Y\to\Omega$ record the omitted coordinate.
For distinct $u,v\in Y$, the projected distance is one or two, and
\begin{equation}\label{eq:proj-rules}
 d(u,v)=1\Rightarrow f(u)\ne f(v),\qquad
 d(u,v)=2\Rightarrow f(u)=f(v).
\end{equation}
Make a graph on $Y$, joining distinct vertices when their distance
is two.  The value of $f$ is constant on each connected component.
Within a component, all distinct projected words therefore have
distance two; by Lemma~\ref{lem:distance2-3}, each component has
at most $m$ vertices.

If there is only one component, we are done.  Otherwise any pair
from different components has projected distance one.
If every component is a singleton, all pairs in $Y$ have distance
one, so all words vary in a single common coordinate and $|Y|\le m$.
Otherwise choose two vertices in one component, normalized to
$000,110$.  A word outside that component has distance one from
both and hence equals either $100$ or $010$.  If just one of
these occurs, say $100$, then any further vertex in the original
component must be $(1,0,c)$ with $c\ne0$.  There is at most one
such vertex because two of them would have distance one.
If both $100$ and $010$ occur, there can be no further vertex in
that component.  In either case $|Y|\le4\le m$.
\end{proof}

\begin{lemma}[Antipodal shell]\label{lem:shell}
If $u,v\in\Omega^3$ have distance three, put
\[
 X(u,v)=\{w\in\Omega^3:d(w,u)=d(w,v)=2\}.
\]
The graph on $X(u,v)$ joining pairs at Hamming distance two is
triangle-free.
\end{lemma}
\begin{proof}
Normalize $u=000$ and $v=111$.  Each word in $X(u,v)$
contains exactly one $0$, exactly one $1$, and one symbol different
from both.  For $x=(0,1,a)$ with $a\notin\{0,1\}$, its neighbors
at distance two in the shell have the forms
\[
 (0,b,1),\qquad(b,1,0),\qquad(1,0,a),
 \quad b\notin\{0,1\}.
\]
Two different neighbors of the first form have distance one, as
do two neighbors of the second form; a neighbor from each of two
different forms has distance three.  Thus no two neighbors of
$x$ are adjacent.  By symmetry this holds for every $x\in X(u,v)$,
so the graph is triangle-free.
\end{proof}

\begin{lemma}[Cross-families in a triangle-free graph]\label{lem:cross}
Let $G$ be triangle-free, and let $B_1,\ldots,B_r$ be nonempty
subsets of $V(G)$, where $r\ge1$.  Suppose that for $i\ne j$
every $x\in B_i$ and $y\in B_j$ are equal or adjacent, and that
\[
 |B_i|\le2m\quad\text{for all }i,
 \qquad |B_i|\le m\quad\text{whenever }B_i
 \text{ is independent}.
\]
Then
\begin{equation}\label{eq:cross}
 \sum_{i=1}^r|B_i|\le2m+2r-2.
\end{equation}
\end{lemma}
\begin{proof}
If some $B_i$ contains an edge $ab$, every vertex of another $B_j$
must be $a$ or $b$: a different vertex adjacent to both would form
a triangle.  The total is therefore at most $2m+2(r-1)$.

Now assume all $B_i$ are independent.  If at most one has size at
least three, the total is at most $m+2(r-1)$.
Otherwise choose $B_i,B_j$ of sizes at least three.  They have
at most one common vertex; two distinct common vertices would be
nonadjacent within either independent set but required to be
adjacent across the two sets.  Choose distinct $a\in B_i$ and
$b\in B_j$, which must be adjacent.  Any vertex in a third set
is $a$ or $b$ by triangle-freeness; since that third set is
independent, it contains at most one such vertex.  Thus the sum
is at most $2m+(r-2)$, proving \eqref{eq:cross}.
\end{proof}

\begin{lemma}[A heavy three-overlap neighbor]\label{lem:heavy-neighbor}
Suppose $S,T\in\calH$ are distinct, $|S\cap T|=3$, and
$\mu(S)>m$.  Then $\mu(T)\le2m$.
\end{lemma}
\begin{proof}
Let $Q=S\cap T$.  Projection of words supported on $S$ to $Q$
is injective.  By Lemma~\ref{lem:projection}, it contains words
$u,v$ at distance three.  A word supported on $T$ has a
three-coordinate projection $w$ satisfying
$d(w,u),d(w,v)\in\{0,2\}$, because the two distinct leaf coordinates
of $S,T$ contribute two to the full distance.  Neither projected
distance can be zero, since $d(u,v)=3$.  Hence $w\in X(u,v)$.

Projection of words on $T$ to $Q$ is also injective.  Label each
projected word by its nonzero symbol on the coordinate $T\setminus Q$.
On a distance-two edge in the shell the two leaf symbols must agree,
otherwise the full distance between the original words would be
three.  Thus every connected component of the induced shell graph
has a constant leaf symbol.  Two vertices in one component must be
adjacent, since their full distance equals their projected distance
and is in $\{2,4\}$, while the projected distance is at most three.
By Lemma~\ref{lem:shell}, each such component has at most two
vertices.  Components with the same leaf symbol would be joined
by an edge, and therefore have different symbols.  There are at
most $m$ components, proving $\mu(T)\le2m$.
\end{proof}

\begin{lemma}[Common three-core]\label{lem:threecore}
Let $Q\subseteq[n]$ have size three, and let $\mathcal T\subseteq\calH$
consist of $s\ge1$ distinct four-supports containing $Q$.  Then
\begin{equation}\label{eq:threecore}
 \sum_{S\in\mathcal T}\mu(S)
 \le\max\{ms,\ m^2+2m+2s-4\}.
\end{equation}
\end{lemma}
\begin{proof}
If $s=1$, use $\mu(S)\le m^2$ from Lemma~\ref{lem:local}.
Suppose $s\ge2$.  Each support in $\mathcal T$ is $Q\cup\{v\}$
for a different leaf $v$.  If every support has multiplicity at
most $m$, then the first term gives the result.
Otherwise choose $S_0\in\mathcal T$ with $\mu(S_0)>m$.
By Lemma~\ref{lem:projection}, the words on $S_0$ project to a pair
$u,v\in\Omega^Q$ at distance three.  For another
support $T\in\mathcal T$, let $B_T$ be the set of restrictions to $Q$
of words supported on $T$.  As in Lemma~\ref{lem:heavy-neighbor},
$B_T\subseteq X(u,v)$, and the graph induced by $B_T$ in the
shell has components of size at most two, with different nonzero
leaf symbols on different components.  In particular,
\[
 |B_T|\le2m,\quad\text{and if }B_T\text{ is independent,}
 \quad |B_T|\le m.
\]
For different supports $T,U\ne S_0$, their leaf coordinates
contribute distance two to any pair of codewords.  Hence the
projected words from $B_T$ and $B_U$ are either equal or at
distance two.  Applying Lemma~\ref{lem:cross} to the $s-1$
sets $B_T$ gives
\[
 \sum_{T\in\mathcal T\setminus\{S_0\}}\mu(T)
 \le2m+2(s-1)-2=2m+2s-4.
\]
Together with $\mu(S_0)\le m^2$, this proves \eqref{eq:threecore}.
\end{proof}

\begin{lemma}[Small three-overlap families]\label{lem:three-small}
Suppose a family $\mathcal T\subseteq\calH$ has at most three
supports, with every two distinct members intersecting in three
coordinates.  Then
\begin{equation}\label{eq:three-small}
 \sum_{S\in\mathcal T}\mu(S)\le m^2+4m.
\end{equation}
If all members of a family $\mathcal T\subseteq\calH$ are
four-subsets of one fixed five-set, then
\begin{equation}\label{eq:five-small}
 \sum_{S\in\mathcal T}\mu(S)\le m^2+8m.
\end{equation}
\end{lemma}
\begin{proof}
For the first statement, if all multiplicities are at most $m$,
the total is at most $3m$.  Otherwise choose a support $S_0$ of
multiplicity greater than $m$.  Lemma~\ref{lem:heavy-neighbor}
bounds each of the at most two other multiplicities by $2m$, while
$\mu(S_0)\le m^2$.

For the second statement, there are at most five different
four-subsets of the fixed five-set, and every two different ones
intersect in three coordinates.  If all multiplicities are at
most $m$, the total is at most $5m$.  Otherwise use a support
of multiplicity greater than $m$ and
Lemma~\ref{lem:heavy-neighbor} on each of the at most four others.
\end{proof}

\subsection{A joint bound for the two layers}

A family of four-supports with a common three-set has a potentially
large weight-four layer.  The next lemma controls it together with the
weight-two layer, rather than bounding the two layers independently.

\begin{lemma}[Joint estimate in the common-three-core case]
\label{lem:joint}
Suppose the support family $\calH$ consists of $s\ge3$ distinct
supports, all of the form $Q\cup\{v\}$ for a fixed three-set $Q$,
and suppose
\begin{equation}\label{eq:linear-three-assumption}
 |C_4|\le ms.
\end{equation}
Then
\begin{equation}\label{eq:joint}
 1+|C_2|+|C_4|\le qn.
\end{equation}
\end{lemma}
\begin{proof}
Let $L\subseteq[n]\setminus Q$ be the set of $s$ occurring leaves.
For $v\in L$, write
\[
 B_v=\{x|_Q:\ x\in C_4,\ \supp(x)=Q\cup\{v\}\}.
\]
Projection onto $Q$ is injective on each fixed support.  Call a
coordinate $i\in Q$ \emph{variable} if some $B_v$ contains two
words having different symbols at $i$, and let $r$ be the number
of variable coordinates, $0\le r\le3$.

Every two-support occurring in $C_2$ must meet $Q$.  Indeed, a
two-set avoiding $Q$ could meet at most two of the $s\ge3$
four-supports and would be disjoint from a third, giving a
forbidden distance six.

Consider weight-two supports $\{i,u\}$ with $i\in Q$ and
$u\notin Q$.  By Lemma~\ref{lem:single} and a leaf $v\ne u$,
each such support has multiplicity one.  If $i$ is variable,
choose a leaf $v_i$ on which the symbol at $i$ varies.  No
support $\{i,u\}$ with $u\ne v_i$ can occur, because it would
meet $Q\cup\{v_i\}$ only at $i$ and have to agree there with
two different codewords.  Thus at most one such support occurs
for each variable coordinate.  There are at most $n-3$ supports
of this kind for a nonvariable coordinate.  The remaining
two-supports lie entirely inside $Q$; there are three of them,
each of multiplicity at most $m$.  Consequently
\begin{equation}\label{eq:C2-r}
 |C_2|\le(3-r)(n-3)+r+3m.
\end{equation}

We next use $r$ to refine the weight-four estimate.

If $r=0$, all $B_v$ are singletons, so $|C_4|=s$.
If $r=1$, relabel the variable coordinate as $1$.  For each
leaf $v$, the two other core symbols are fixed throughout $B_v$;
call this ordered pair $c_v$, and let $A_v\subseteq\Omega$ be the
set of values at coordinate $1$.  Then $|A_v|=|B_v|$.  For two
leaves $v,w$, let $t=d(c_v,c_w)\in\{0,1,2\}$.  Since the two
leaf coordinates contribute two to the full distance, compatibility
of all word pairs is equivalent to
\[
 t+\mathbf1_{\{a\ne b\}}\in\{0,2\}
 \qquad(a\in A_v,\ b\in A_w).
\]
When $t=0$ or $2$, both sets must be the same singleton; when
$t=1$, they are disjoint.  Thus the sets belonging to leaves of
multiplicity at least two are pairwise disjoint.  Their total
size is at most $m$, and every remaining leaf contributes one.
As $r=1$ guarantees at least one nonsingleton leaf,
\begin{equation}\label{eq:r1}
 |C_4|\le m+s-1.
\end{equation}
If $r=2$, the third coordinate of $Q$ is constant on each
individual $B_v$.  For each leaf, delete that fixed coordinate
from its four-support; one obtains a length-three code with
mutual distance exactly two.  Lemma~\ref{lem:distance2-3}
then gives $|B_v|\le m$, so $|C_4|\le ms$.
For $r=3$ we use the assumption \eqref{eq:linear-three-assumption}.

Since $s\le n-3$, the four cases combine with \eqref{eq:C2-r}
as follows:
\[
\begin{array}{c|c|c|c}
 r&|C_2|\text{ at most}&|C_4|\text{ at most}&
     1+|C_2|+|C_4|\text{ at most}\\ \hline
 0&3(n-3)+3m&s&4n+3m-11\\
 1&2(n-3)+1+3m&m+s-1&3n+4m-8\\
 2&(n-3)+2+3m&ms&qn\\
 3&3+3m&ms&mn+4
\end{array}
\]
The first two last-column expressions are at most $qn=(m+1)n$
when $m\ge4$ and $n\ge10$.  In the last row,
$mn+4\le(m+1)n$ because $n\ge4$.
This proves \eqref{eq:joint}.
\end{proof}

\section{Structural alternatives for four-support families}\label{sec:structure}

We need two standard facts about intersections and matchings.
The \emph{matching number} $\nu(G)$ of a graph is the largest
number of pairwise disjoint edges.

\begin{lemma}[Three-intersecting four-sets]\label{lem:3int4}
If $\mathcal F\subseteq\binom{[n]}4$ has
$|S\cap T|\ge3$ for all $S,T\in\mathcal F$, then either
all members contain one fixed three-set, or every member is
contained in one fixed five-set.
\end{lemma}
\begin{proof}
The assertion is immediate for at most one member.  Otherwise
choose distinct $S=P\cup\{a\}$ and $T=P\cup\{b\}$, with
$|P|=3$ and $a\ne b$.  If every member contains $P$, we are done.
Otherwise choose $U$ not containing $P$.  Its intersections with
$S,T$ force
$U=(P\setminus\{p\})\cup\{a,b\}$ for some $p\in P$.
For any other $V$, if $P\subseteq V$, write $V=P\cup\{c\}$:
$|V\cap U|\ge3$ forces $c\in\{a,b\}$.  If $P\nsubseteq V$,
intersections with $S,T$ force
$V=(P\setminus\{p'\})\cup\{a,b\}$ for some $p'\in P$.
Thus every member lies in $P\cup\{a,b\}$.
\end{proof}

We use the following specialization of the complete nontrivial-intersection
theorem of Ahlswede and Khachatrian~\cite{AK96}: for $n\ge10$, a
$2$-intersecting family of four-subsets of $[n]$ with no common
two-set has size at most
\begin{equation}\label{eq:AK}
 4n-15.
\end{equation}
We also use the Erd\H{o}s--Gallai matching theorem~\cite{EG59}
in the form: a simple graph on $N\ge7$ vertices without a matching
of size three has at most $2N-3$ edges.

\begin{lemma}[No common two-core]\label{lem:no-core}
Suppose $n\ge10$, and the $2$-intersecting support family $\calH$
has no common two-set.  Then
\begin{equation}\label{eq:no-core}
 |C_4|\le m^2+4n+10m-21.
\end{equation}
\end{lemma}
\begin{proof}
If every two distinct members of $\calH$ meet in at least three
positions, Lemma~\ref{lem:3int4} and the absence of a common
two-set show that all members are four-subsets of one five-set.
Lemma~\ref{lem:three-small} then gives
$|C_4|\le m^2+8m$, smaller than \eqref{eq:no-core}.

Otherwise choose $S,T\in\calH$ with $P=S\cap T$ of size two, and
write
\begin{equation}\label{eq:SAT}
 S=P\mathbin{\dot\cup}A,\qquad
 T=P\mathbin{\dot\cup}B,
 \qquad |P|=|A|=|B|=2,
\end{equation}
where $P,A,B$ are pairwise disjoint.  Since $P$ is not common,
fix $R\in\calH$ with $P\nsubseteq R$.  Partition $\calH$
into $\calH_0$ (supports not containing $P$) and $\calH_1$
(supports containing $P$).  For a subfamily $\mathcal F$ define
\[
 \operatorname{exc}(\mathcal F)=
 \sum_{U\in\mathcal F}(\mu(U)-1).
\]

We first claim
\begin{equation}\label{eq:H0-exc}
 \operatorname{exc}(\calH_0)\le4(m-1).
\end{equation}
Let $U\in\calH_0$ have $\mu(U)>1$.  Then
$|U\cap S|,|U\cap T|\ge2$ and
\[
 |U\cap S|+|U\cap T|
 =|U\cap(S\cup T)|+|U\cap P|\le5.
\]
If both intersections had size two, they would be different
pairs (since $P\nsubseteq U$), contradicting the second assertion
of Lemma~\ref{lem:overlap}.  Thus the sizes are $(2,3)$ or
$(3,2)$, and $U$ is of one of the forms
\begin{equation}\label{eq:types}
 U=B\cup\{p,a\},\quad p\in P,\ a\in A;
 \qquad\text{or}\qquad
 U=A\cup\{p,b\},\quad p\in P,\ b\in B.
\end{equation}
Among the four supports of the first form, at most two can have
multiplicity greater than one.  In fact, if two of them use
different elements of $P$ and different elements of $A$, they
intersect exactly in $B$; for either support this would give a
second exact-overlap pair distinct from its intersection with $S$.
Thus two opposite cells of the $2\times2$ choice grid cannot both
have multiplicity greater than one; any set of three cells contains
such a pair.  The second form obeys the same argument.
Altogether at most four supports in $\calH_0$ can have
multiplicity greater than one, each of multiplicity at most $m$
by Lemma~\ref{lem:overlap}.  This proves \eqref{eq:H0-exc}.

Represent $\calH_1$ by a graph $G$ on $[n]\setminus P$:
$e\in E(G)$ when $P\cup e\in\calH_1$.  In particular, $A,B\in E(G)$
and $A\cap B=\varnothing$.  We shall use this fixed pair of
disjoint edges rather than subdividing into matching-free cases.
Let $r=|R\cap P|\in\{0,1\}$.

If $r=0$, every $P\cup e\in\calH_1$ must meet $R$ in two
coordinates, so $e\subseteq R$.  Hence there are at most six
such supports, all of multiplicity at most $m$ by
Lemma~\ref{lem:overlap}.  Using \eqref{eq:AK} and
\eqref{eq:H0-exc},
\[
 |C_4|\le|\calH|+4(m-1)+6(m-1)
 \le4n+10m-25\le m^2+4n+10m-21.
\]

If $r=1$, write $R\cap P=\{p\}$ and $D=R\setminus P$, with
$|D|=3$.  Every edge of $G$ meets $D$.  Call an edge
\emph{exceptional} if it lies wholly in $D$, and
\emph{ordinary} otherwise.  There are at most three exceptional
edges.  An ordinary edge corresponds to a support whose exact
intersection with $R$ has size two; its multiplicity is therefore
at most $m$.

An ordinary edge $e$ of multiplicity greater than one must meet
every other edge of $G$.  Otherwise a disjoint edge $f$ would
produce the exact-overlap pair $P$ on the support $P\cup e$,
distinct from its exact-overlap pair with $R$; then
Lemma~\ref{lem:overlap} would force multiplicity one.
In particular, such an edge must meet both disjoint edges $A,B$.
Thus it is one of the four cross-edges between $A$ and $B$.
Any two edges that meet every edge of $G$ must meet each other.
Since a pairwise-intersecting subfamily of the four edges of
$K_{2,2}$ has size at most two, the ordinary excess is at most
$2(m-1)$.

The exceptional edges correspond to at most three four-supports,
any two of which intersect in three coordinates (they all contain
$P$ and the exceptional edges are two-subsets of $D$).
Lemma~\ref{lem:three-small} bounds the \emph{total multiplicity}
on these supports by $m^2+4m$.  Combining this with
\eqref{eq:H0-exc} and the baseline one word per support, and
allowing a harmless overcount of the exceptional baseline, gives
\[
 \begin{aligned}
 |C_4|&\le |\calH|+4(m-1)+2(m-1)+(m^2+4m)\\
 &\le(4n-15)+6(m-1)+m^2+4m\\
 &=m^2+4n+10m-21.
 \end{aligned}
\]
This completes the proof.
\end{proof}

We next treat the case in which all four-supports have a common
two-set $P$.  Each support then has a unique representation
$P\cup e$ for an edge $e$ in an \emph{external graph} $G$ on
$[n]\setminus P$.

\begin{lemma}[Three disjoint external edges]\label{lem:matching-three}
Suppose the external graph $G$ contains a matching of size three
and $n-2\ge q$.  Then
\begin{equation}\label{eq:matching-three}
 |C_4|\le\binom{n-2}{2},\qquad |C_2|\le2n-3.
\end{equation}
\end{lemma}
\begin{proof}
Let $e_1,e_2,e_3$ be pairwise disjoint edges of $G$ and choose
codewords on the supports $P\cup e_i$.  For two such words, the
four external coordinates already contribute distance four, so
the words have identical symbols on $P$.  Every other word in
$C_4$ has an external two-support disjoint from at least one of
the $e_i$, and consequently it has the same symbols on $P$.
Puncturing the two coordinates of $P$ is therefore an injective
isometry on $C_4$, whose image consists of weight-two words of
length $n-2$.  Lemma~\ref{lem:pure} gives the first bound.

Every weight-two support must meet all three of $P\cup e_i$;
a two-set avoiding $P$ cannot meet all three disjoint edges.
Hence every occurring two-support contains a coordinate of $P$.
Write $P=\{a,b\}$.  The possible supports are $P$ and
$\{a,v\},\{b,v\}$ for $v\notin P$.  On a support $\{a,v\}$,
choose $e_i$ not containing $v$.  Lemma~\ref{lem:single}
fixes the symbol at $a$ by comparison with a four-word on
$P\cup e_i$, so this support has multiplicity one.  The same
holds for $\{b,v\}$.
If any noncore two-support occurs, comparison with a word on it
fixes one symbol for all words on $P$, and so $P$ has
multiplicity at most one.  This gives at most $2(n-2)+1=2n-3$
words.  If none occurs, $C_2$ is supported only on $P$ and
has size at most $m\le2n-3$ (as $n-2\ge q$).
\end{proof}

\begin{lemma}[A matching of size two]\label{lem:matching-two}
Suppose the external graph $G$ has a matching of size two but no
matching of size three, and $n\ge10$.  Then
\begin{equation}\label{eq:matching-two}
 |C_4|\le m^2+2n+4m-9.
\end{equation}
\end{lemma}
\begin{proof}
Fix disjoint edges $e,f$ in $G$, and let $X$ be the set of edges
meeting both $e$ and $f$.  Every edge of $X$ joins one endpoint
of $e$ to one endpoint of $f$, so $X$ is a subset of the
four-edge graph $K_{2,2}$.  Put $E_0=E(G)\setminus X$.

Words on $P\cup e$ and $P\cup f$ agree on $P$ by exact-overlap
synchronization.  Every edge of $E_0$ is disjoint from at least
one of $e,f$, so every codeword supported on $P\cup g$ with
$g\in E_0$ has those same symbols on $P$.
Puncturing $P$ maps this part of $C_4$ isometrically to a
weight-two code whose distinct supports are exactly $E_0$.
That graph has at most two isolated edges; three isolated edges
would be a matching of size three in $G$.  By the multiplicity
part of Lemma~\ref{lem:pure}, this portion contributes at most
\begin{equation}\label{eq:E0}
 |E_0|+2(m-1)
\end{equation}
words.

We claim that the supports corresponding to $X$ contribute
at most
\begin{equation}\label{eq:X}
 m^2+2m
\end{equation}
words.  We verify all possibilities, so no previous-case analysis
is needed.  For $|X|=0$ the assertion is empty, and for
$|X|=1$ it follows from Lemma~\ref{lem:local}.
For $|X|=2$, either the two edges are disjoint, in which case
each has multiplicity at most $m$ by Lemma~\ref{lem:overlap},
or they share a vertex, in which case their four-supports
contain a common three-set and Lemma~\ref{lem:threecore}
gives total multiplicity at most
$\max\{2m,m^2+2m\}=m^2+2m$.
For $|X|=3$, two cross-edges are opposite and therefore
have multiplicity at most $m$ each, while the third has
multiplicity at most $m^2$: the sum is at most $m^2+2m$.
For $|X|=4$, every cross-edge has a disjoint opposite
cross-edge, so the sum is at most $4m\le m^2+2m$ because $m\ge4$.
This proves \eqref{eq:X}.

The Erd\H{o}s--Gallai theorem, applied to the $n-2\ge8$
external vertices, yields $|E(G)|\le2(n-2)-3=2n-7$.
Adding \eqref{eq:E0} and \eqref{eq:X},
\[
 |C_4|\le |E_0|+2(m-1)+m^2+2m
 \le m^2+2n+4m-9.
\]
\end{proof}

\begin{lemma}[No two disjoint external edges]\label{lem:matching-one}
Suppose the external graph $G$ has no matching of size two.
Then either all its edges share a vertex, or $G$ is contained
in a triangle.  In the triangle case,
\begin{equation}\label{eq:triangle}
 |C_4|\le m^2+4m.
\end{equation}
If $G$ is a star with $s$ edges, all four-supports contain
a common three-set, and the estimate \eqref{eq:threecore}
holds with this $s$.
\end{lemma}
\begin{proof}
For the graph assertion, take two incident edges $ab,ac$.
If an edge avoids $a$, it must be $bc$ to meet both.
Any edge meeting all three of $ab,ac,bc$ is one of these
three.  The cases with zero or one edge are immediate.
When $G$ is contained in a triangle, there are at most three
four-supports, any two of which intersect in three coordinates.
Lemma~\ref{lem:three-small} gives \eqref{eq:triangle}.
If $G$ is a star with center $c$, all four-supports contain
$P\cup\{c\}$; apply Lemma~\ref{lem:threecore}.
\end{proof}

\begin{proposition}[Structural alternatives]\label{prop:alternatives}
Suppose $n\ge10$ and $C_4\ne\varnothing$.  At least one of the
following holds:
\begin{enumerate}[label=\textup{(\roman*)}]
\item\label{item:matching} All four-supports contain a common two-set
and their external graph has a matching of size at least three.
\item\label{item:finite} The weight-four layer satisfies
\begin{equation}\label{eq:finite-bound}
 |C_4|\le m^2+4n+10m-21.
\end{equation}
\item\label{item:three} All four-supports contain a common three-set,
with $s\ge3$ distinct supports, and
\[
 |C_4|\le\max\{ms,m^2+2m+2s-4\}.
\]
\end{enumerate}
\end{proposition}
\begin{proof}
By Lemma~\ref{lem:2int}, $\calH$ is $2$-intersecting.
If it has no common two-set, Lemma~\ref{lem:no-core} gives (ii).
Otherwise let $P$ be a common two-set and form the external graph.
A matching of size three gives (i).  If the matching number is
two, Lemma~\ref{lem:matching-two} gives (ii), because
\[
 m^2+2n+4m-9\le m^2+4n+10m-21
 \qquad(n\ge10,\ m\ge4).
\]
If the matching number is at most one, use
Lemma~\ref{lem:matching-one}.  Its triangle case satisfies (ii).
In the star case, if $s\ge3$, Lemma~\ref{lem:threecore}
gives (iii).  If $s\le2$, the same lemma gives
$|C_4|\le m^2+2m$, which satisfies (ii).
\end{proof}

\section{Proofs of the main results}\label{sec:proof}

\begin{proof}[Proof of Theorem~\ref{thm:main}]
The construction in \eqref{eq:standard-lower} gives the lower
bound.  We prove the upper bound.  Let $q\ge5$ be an arbitrary
integer, let $m=q-1$, and take $n\ge N(q)$.
Normalize $\calC$ as in \eqref{eq:layers}.  If $C_4=\varnothing$,
Lemma~\ref{lem:pure} gives $|\calC|\le1+\binom n2$, since
$n\ge2q+1\ge q$.

Suppose $C_4\ne\varnothing$.  By Lemma~\ref{lem:C2-linear},
\begin{equation}\label{eq:mainC2}
 |C_2|\le4n+4m-14.
\end{equation}
Use Proposition~\ref{prop:alternatives}.

In alternative~\ref{item:matching}, Lemma~\ref{lem:matching-three}
applies because $n-2\ge q$.  Hence
\[
 |\calC|\le1+(2n-3)+\binom{n-2}{2}=1+\binom n2.
\]

In alternative~\ref{item:finite}, we have
\[
 \begin{aligned}
 |\calC|&=1+|C_2|+|C_4|\\
 &\le1+(4n+4m-14)+(m^2+4n+10m-21)\\
 &=m^2+8n+14m-34.
 \end{aligned}
\]
This is at most $1+\binom n2$ exactly when
\begin{equation}\label{eq:quadratic}
 n^2-17n-2m^2-28m+70\ge0.
\end{equation}
After substituting $m=q-1$, the larger root is
\[
 \frac{17+\sqrt{8q^2+96q-95}}2.
\]
Thus \eqref{eq:quadratic} holds by the definition of $N(q)$.

Finally, consider alternative~\ref{item:three}.  If
$|C_4|\le ms$, the joint estimate in Lemma~\ref{lem:joint}
gives $|\calC|\le qn$.  For $n\ge2q+1$,
\[
 1+\binom n2-qn
 =\frac{n^2-(2q+1)n+2}{2}\ge0,
\]
so the desired upper bound follows.
If $|C_4|>ms$, the second term of the maximum in
alternative~\ref{item:three} gives
\[
 |C_4|\le m^2+2m+2s-4
 \le m^2+2n+2m-10,
\]
since $s\le n-3$.  Therefore, by \eqref{eq:mainC2},
\[
 |\calC|\le m^2+6n+6m-23
 \le m^2+8n+14m-34,
\]
where the last inequality holds for $n\ge10,m\ge4$.
The right-hand side has already been bounded by $1+\binom n2$
using \eqref{eq:quadratic}.

All three alternatives are covered, proving \eqref{eq:main}.
To verify the final assertion, note that
\[
 \frac{17+\sqrt{8q^2+96q-95}}2\le2q+1
 \quad\Longleftrightarrow\quad q^2-27q+40\ge0,
\]
which holds for every integer $q\ge26$.  Consequently
$N(q)=2q+1$ in this range.
\end{proof}

\begin{proof}[Proof of Theorem~\ref{thm:threshold}]
Fix two distinct symbols $0,1\in\Sigma$ and choose
$Q=\{1,2,3\}\subseteq[n]$.  For $v\in\{4,\ldots,n\}$ and
$a\in\Sigma\setminus\{0\}$ define a word $x_{v,a}$ by
\[
 (x_{v,a})_1=(x_{v,a})_2=a,\qquad
 (x_{v,a})_3=(x_{v,a})_v=1,
\]
with every other coordinate zero.  For $b\ne0$, let $y_b$ be
the weight-two word with values $b,b$ on coordinates $1,2$.
For $u\in\{4,\ldots,n\}$, let $z_u$ be the weight-two word
with values $1,1$ on coordinates $3,u$.  Put
\[
 \mathcal D_{q,n}=\{0\}\cup\{x_{v,a}:v\ge4,\ a\ne0\}
                     \cup\{y_b:b\ne0\}\cup\{z_u:u\ge4\}.
\]
All parts of this union are disjoint.  We check that the
nonzero distances are exactly of the allowed types.
For $x_{v,a},x_{w,b}$, the distance is two if either
$v=w$ with $a\ne b$, or $v\ne w$ with $a=b$;
it is four if $v\ne w$ and $a\ne b$.
Two different $y$-words have distance two, as do two different
$z$-words, and every $y$-word is at distance four from every
$z$-word.  Moreover,
\[
 d(y_b,x_{v,a})=
 \begin{cases}2,&a=b,\\4,&a\ne b,\end{cases}
 \qquad
 d(z_u,x_{v,a})=
 \begin{cases}2,&u=v,\\4,&u\ne v.\end{cases}
\]
Every nonzero codeword has weight two or four, so its distance
from zero is also allowed.  Therefore $\mathcal D_{q,n}$ is a
$\{2,4\}$-distance code of size
\begin{equation}\label{eq:construction}
 |\mathcal D_{q,n}|
 =1+m(n-3)+m+(n-3)=q(n-2).
\end{equation}

At $n=2q-2$ (which is at least $8$ for $q\ge5$),
\[
 q(n-2)-\left(1+\binom n2\right)=q-4>0.
\]
Thus \eqref{eq:main} fails at $n=2q-2$, and by the
definition of the eventual threshold $N_*(q)\ge2q-1$.
Theorem~\ref{thm:main} proves $N_*(q)\le N(q)$, and
$N(q)=2q+1$ for $q\ge26$.
The bracket \eqref{eq:bracket} follows.
\end{proof}

\begin{remark}[Scope of the method]
The proof is a structural argument for arbitrary alphabet sizes
$q\ge5$; it does not use finite-field addition, a linear code
hypothesis, or exhaustive computation.  It establishes the least
eventual length threshold within an additive error of two for
$q\ge26$, without determining which of the three integers
$2q-1,2q,2q+1$ is the exact threshold.
\end{remark}

\section*{Declaration of generative AI and AI-assisted technologies in the manuscript preparation process}
OpenAI ChatGPT was used to assist with organization, language editing,
preliminary proof checking, and preparation of LaTeX source.  The author(s)
are responsible for independently reviewing the mathematical arguments,
checking all bibliographic information, and approving the final submission.


\begin{thebibliography}{99}

\bibitem{AK96}
R.~Ahlswede and L.~H.~Khachatrian,
The complete nontrivial-intersection theorem for systems of finite sets,
J. Combin. Theory Ser. A \textbf{76} (1996), 121--138.

\bibitem{AK97}
R.~Ahlswede and L.~H.~Khachatrian,
The complete intersection theorem for systems of finite sets,
European J. Combin. \textbf{18} (1997), 125--136.

\bibitem{BannaiIto84}
E.~Bannai and T.~Ito,
\emph{Algebraic Combinatorics I: Association Schemes},
Benjamin/Cummings, 1984.

\bibitem{Barg24}
A.~Barg, A.~Glazyrin, W.-J.~Kao, C.-Y.~Lai, P.-C.~Tseng,
and W.-H.~Yu,
On the size of maximal binary codes with 2, 3, and 4 distances,
Combinatorial Theory \textbf{4} (2024), no.~1, Paper 7.
\newblock \url{https://doi.org/10.5070/C64163844}.

\bibitem{BDZZ21}
P.~Boyvalenkov, K.~Delchev, D.~V.~Zinoviev, and V.~A.~Zinoviev,
On two-weight codes,
Discrete Math. \textbf{344} (2021), 112318.
\newblock \url{https://doi.org/10.1016/j.disc.2021.112318}.

\bibitem{BCN89}
A.~E.~Brouwer, A.~M.~Cohen, and A.~Neumaier,
\emph{Distance-Regular Graphs}, Springer, 1989.

\bibitem{CK86}
R.~Calderbank and W.~M.~Kantor,
The geometry of two-weight codes,
Bull. London Math. Soc. \textbf{18} (1986), 97--122.

\bibitem{Delsarte73}
P.~Delsarte,
\emph{An Algebraic Approach to the Association Schemes of Coding Theory},
Philips Res. Rep. Suppl. \textbf{10} (1973), 1--97.

\bibitem{DGS77}
P.~Delsarte, J.-M.~Goethals, and J.~J.~Seidel,
Spherical codes and designs,
Geom. Dedicata \textbf{6} (1977), 363--388.

\bibitem{EFIN87}
H.~Enomoto, P.~Frankl, N.~Ito, and K.~Nomura,
Codes with given distances,
Graphs Combin. \textbf{3} (1987), 25--38.

\bibitem{EG59}
P.~Erd\H{o}s and T.~Gallai,
On maximal paths and circuits of graphs,
Acta Math. Acad. Sci. Hungar. \textbf{10} (1959), 337--356.

\bibitem{EKR61}
P.~Erd\H{o}s, C.~Ko, and R.~Rado,
Intersection theorems for systems of finite sets,
Quart. J. Math. Oxford Ser. (2) \textbf{12} (1961), 313--320.

\bibitem{Frankl78}
P.~Frankl,
On intersecting families of finite sets,
J. Combin. Theory Ser. A \textbf{24} (1978), 146--161.

\bibitem{FW81}
P.~Frankl and R.~M.~Wilson,
Intersection theorems with geometric consequences,
Combinatorica \textbf{1} (1981), 357--368.

\bibitem{FT18}
P.~Frankl and N.~Tokushige,
\emph{Extremal Problems for Finite Sets},
American Mathematical Society, 2018.

\bibitem{GST06}
D.~C.~Gijswijt, A.~Schrijver, and H.~Tanaka,
New upper bounds for nonbinary codes based on the Terwilliger algebra and
semidefinite programming,
J. Combin. Theory Ser. A \textbf{113} (2006), 1719--1731.

\bibitem{Hamming50}
R.~W.~Hamming,
Error detecting and error correcting codes,
Bell System Tech. J. \textbf{29} (1950), 147--160.

\bibitem{HP03}
W.~C.~Huffman and V.~Pless,
\emph{Fundamentals of Error-Correcting Codes},
Cambridge University Press, 2003.

\bibitem{HM67}
A.~J.~W.~Hilton and E.~C.~Milner,
Some intersection theorems for systems of finite sets,
Quart. J. Math. Oxford Ser. (2) \textbf{18} (1967), 369--384.

\bibitem{LovaszPlummer86}
L.~Lov\'asz and M.~D.~Plummer,
\emph{Matching Theory}, North-Holland, 1986.

\bibitem{LRV23}
I.~Landjev, A.~Rousseva, and K.~Vorobev,
Constructions of binary codes with two distances,
Discrete Math. \textbf{346} (2023), 113337.

\bibitem{MS77}
F.~J.~MacWilliams and N.~J.~A.~Sloane,
\emph{The Theory of Error-Correcting Codes},
North-Holland, 1977.

\bibitem{Sch05}
A.~Schrijver,
New code upper bounds from the Terwilliger algebra and semidefinite
programming,
IEEE Trans. Inform. Theory \textbf{51} (2005), 2859--2866.

\bibitem{vanLint99}
J.~H.~van Lint,
\emph{Introduction to Coding Theory}, 3rd ed., Springer, 1999.

\end{thebibliography}
\end{document}